\documentclass[letterpaper, 10 pt, conference]{ieeeconf}  

\IEEEoverridecommandlockouts                              

\usepackage{amssymb}
\usepackage{amsthm}
\usepackage{graphics} 
\usepackage{float}
\usepackage{epsfig} 
\usepackage{times} 
\usepackage{amsmath, bm} 
\usepackage{amsfonts}  
\usepackage{comment}
\usepackage{xcolor}
\usepackage{cite}
\usepackage{hyperref}
\usepackage[english]{babel}

\newcommand{\cC}{{\cal C}}

\newcommand{\cD}{{\cal D}}
\newcommand{\cL}{{\cal L}}

\newcommand{\mC}{{\mathbb C}}

\newcommand{\mR}{{\mathbb R}}

\newcommand{\bU}{{\mathbf U}}
\newcommand{\mU}{{\mathbb U}}
\newcommand{\bG}{{\mathbf G}}
\newcommand{\bD}{{\mathbf D}}

\newcommand{\bs}{{\mathbf s}}

\newcommand{\bP}{{\mathbf P}}

\newcommand{\bM}{{\mathbf M}}

\newcommand{\by}{{\mathbf y}}
\newcommand{\bff}{{\mathbf f}}

\newcommand{\bw}{{\mathbf w}}
\newcommand{\bx}{{\mathbf x}}

\newcommand{\bC}{{\boldsymbol{\mathcal C}}}

\newcommand{\bX}{{\mathbf X}}

\newcommand{\bA}{{\mathbf A}}
\newcommand{\bB}{{\mathbf B}}

\newcommand{\bd}{{\mathbf d}}
\newcommand{\bz}{{\mathbf z}}

\newcommand{\bq}{{\mathbf q}}

\theoremstyle{plain}
\newtheorem{theorem}{\textbf{Theorem}}
\newtheorem{definition}{\textbf{Definition}}

\newtheorem{corollary}{\textbf{Corollary}}
\theoremstyle{definition}
\newtheorem{assumption}{Assumption}
\newtheorem{remark}{\textbf{Remark}}

\title{\LARGE \bf
Path-Integral Formula for Computing Koopman Eigenfunctions 
}
\author{Shankar A. Deka$^{*}$, Sriram S.K.S. Narayanan$^{*}$ and Umesh Vaidya
\thanks{* These authors contributed equally to this paper.}
\thanks{UV will like to acknowledge the financial support from NSF grant 1932458. Shankar A. Deka is with the Decision
and Control Systems division in the Dept. of Electrical Engr.
and Computer Science, KTH Royal Institute of Technology, Sweden. Email: deka@kth.se. Sriram S.K.S Narayanan and Umesh Vaidya are with the Dept. of Mechanical Engr., Clemson University, Clemson SC. Email: (sriramk,uvaidya)@clemson.edu
}
}

\begin{document}
\maketitle
\thispagestyle{empty}
\pagestyle{empty}



\begin{abstract}
The paper is about the computation of the principal spectrum of the Koopman operator (i.e., eigenvalues and eigenfunctions). The principal eigenfunctions of the Koopman operator are the ones with the corresponding eigenvalues equal to the eigenvalues of the linearization of the nonlinear system at an equilibrium point. 
The main contribution of this paper is to provide a novel approach for computing the principal eigenfunctions using a path-integral formula. Furthermore, we provide conditions based on the stability property of the dynamical system and the eigenvalues of the linearization towards computing the principal eigenfunction using the path-integral formula.
Further, we provide a Deep Neural Network framework \textcolor{black}{that utilizes our proposed path-integral approach} for eigenfunction computation in high-dimension systems. Finally, we present simulation results for the computation of principal eigenfunction and demonstrate their application for determining the stable and unstable manifolds and constructing the Lyapunov function. 
\end{abstract}

\section{Introduction}
The Koopman operator theory is emerging as a powerful tool for the analysis and synthesis of nonlinear systems \cite{mezic2020spectrum,huang2020data,huang2019optimal,huang2022convex,korda2018convergence,moyalan2023data,sinha2019computation}. The linear lifting of a nonlinear system provided by the Koopman operator in the space of functions is successfully exploited for control design \cite{korda2018linear,vaidya_Koopmanspectrumcdc2022}, prediction \cite{kaiser2018sparse,brunton2016koopman}, and uncertainty propagation \cite{matavalam2020data,sinha2016operator} in a dynamical system. However, the spectral properties, i.e., the eigenvalues and eigenfunctions, of the Koopman operator still need to be explored, especially for control \cite{korda2020optimal,vaidya_Koopmanspectrumcdc2022}.


In this paper, we are specifically interested in identifying the principal eigenfunctions of the Koopman operator. The principal eigenfunctions are associated with the eigenvalues of the linearization of the nonlinear system at an equilibrium point. The principal eigenfunctions provide a powerful tool for analyzing and synthesizing controllers for nonlinear systems. These eigenfunctions can be used as a change of coordinates for the linear representation of a nonlinear system over a large region of the state space \cite{lan2013linearization,mezic2020spectrum}. The extent of validity of these eigenfunctions determines the size of the domain over which the linear representation is valid. For example, in a system with a stable equilibrium point, these eigenfunctions are well defined in the domain of attraction of the equilibrium point. The zero-level curves of the eigenfunction are used to identify the stable and unstable manifolds of the dynamical system. More recently, the connection between the principal eigenfunctions of the Koopman operator and the solution of the Hamilton Jacobi equation has been established \cite{vaidya_Koopmanspectrumcdc2022}. This connection provides a systematic approach for formulating and solving various control problems, including optimal control, robust control, and input-output gain analysis of a nonlinear system \cite{SarangACC2023}. For all these reasons, it becomes imperative to develop systematic and robust computational methods for determining the principal spectrum of the Koopman operator. In \cite{mauroy2016global}, Taylor and Bernstein's polynomials were used to approximate the eigenfunctions. To reduce the computation cost for high dimensional systems, \cite{schlosser2022sparsity} proposed to decompose the system as a set of interconnected systems and exploit its sparsity structure. A convex formulation to approximate the principal eigenfunctions is provided in \cite{umathe2022reachability}. However, these methods cannot be easily extended to a general high-dimensional system.

The main contribution of this paper is to provide a novel approach for the computation of the principal eigenfunctions of the Koopman operator. The approach relies on decomposing principal eigenfunctions into linear and purely nonlinear parts. The linear part of the eigenfunction is obtained as the left eigenvector of the linearization of system dynamics at the equilibrium point. The nonlinear part is shown to satisfy a linear partial differential equation (PDE). The solution of this linear PDE is obtained using a path-integral formulation. In particular, the value of the eigenfunction at any given point $\bx_0$, is obtained by integrating a known function along the system trajectory forward in time with  $\bx_0$ as the initial state. We provide conditions based on the stability properties of the system for the path-integral formula to work. The path-integral approach does not involve a choice of basis function, making it attractive for complex systems. Furthermore, we present a DNN framework to approximate the solution of the PDE for high-dimensional systems. Finally, we demonstrate the application of the developed framework for the computation of stable/unstable manifolds and the construction of Lyapunov functions. 

\section{Preliminaries and Notations}
Consider the continuous-time dynamical system 
\begin{align}
\dot \bx=\bff(\bx),\;\;\;\bx\in \bX\subset \mR^n. \label{sys}
\end{align}
The following assumption is made on the vector field in the rest of the paper.

\begin{assumption}\label{assume_sys}
We assume that the vector field $\bff(\bx)$ is at least $\cC^2(\bX)$ \textcolor{black}{(twice continuously differentiable)} and $\bx=0$ is a hyperbolic equilibrium point of the system, i.e., $\bA:=\frac{\partial \bff}{\partial \bx}(0)$ has no eigenvalues on the imaginary axis. 
\end{assumption}
 
\begin{definition} [Koopman Operator] \textcolor{black}{Let $\bs_t(\bx)$ be the solution of the dynamical system  (\ref{sys}) at time $t$ starting from the initial condition $\bx$.} The Koopman operator $\mathbb{U}_t:\cL_{\infty}(\bX)\to\cL_\infty(\bX)$ associated with the dynamical system (\ref{sys}) is defined as 
\begin{align}\label{eq:koop_def}
[\mathbb{U}_t \psi](\bx)=\psi(\bs_t(\bx)),
\end{align}
where $\psi$ \textcolor{black}{(commonly referred to as an observable function) is} defined on $\cL_\infty(\bX)$, which is the space of essentially bounded functions on $\bX$. The infinitesimal generator ${\cal K}_{\bf f}$ for the Koopman operator is given by
\begin{align}
\lim_{t\to 0}\frac{(\mathbb{U}_t-I)\psi}{t}=\frac{\partial \psi}{\partial \bx}{\bf f}(\bx)=:{\cal K}_{\bf f} \psi,\;\;t\geq 0. \label{K_generator}
\end{align}
\end{definition}

\begin{definition}[Eigenvalues and Eigenfunctions] A function $\phi(\bx)\in \cC^1(\bX)$  is said to be an eigenfunction of the Koopman operator associated with eigenvalue $\lambda$ if
\begin{align}
[\mU_t \phi](\bx)=e^{\lambda t}\phi(\bx),\;\;\textcolor{black}{t\geq 0}.\label{eig_koopman}
\end{align}
Using the Koopman generator, equation (\ref{eig_koopman}) can be written as 
\begin{align}
    \textcolor{black}{{\cal K}_{\bf f}\phi =} \frac{\partial \phi}{\partial \bx}{\bf f}(\bx)=\lambda \phi(\bx).\label{eig_koopmang}
\end{align}
\end{definition}

\noindent Notice that equations (\ref{eig_koopman}) and (\ref{eig_koopmang}) \textcolor{black}{provide a ``global'' definition of Koopman spectrum in the sense that it holds for all $t\in [0,\infty)$ and all $x\in \bX$}. However, the spectrum can be defined over finite time or over a subset of the state space and is of interest to us in this paper. Furthermore, in this paper, we are also interested in computing the spectrum associated with the eigenvalues of the linearization of the nonlinear system at an equilibrium point.  
\begin{definition}[Open Eigenfunction \cite{mezic2020spectrum}]
 Let $\phi: \bC\to \mC$, where $\bC\subset \bX$ is not an invariant set. Let $\bx\in  \bC$, and
$\tau \in (\tau^-(\bx),\tau^+(\bx))= I_\bx$, a connected open interval such that \textcolor{black}{$\bs_\tau (\bx) \in \bC$} for all  $\tau \in I_\bx$.
If
\[[\mU_\tau \phi](\bx) = \phi(\bs_\tau(\bx)) =e^{\lambda \tau}  \phi(\bx)\;\;\;\;\forall \tau \in I_\bx, \]
then \textcolor{black}{$\phi$} is called \textcolor{black}{an} open eigenfunction of the Koopman operator family $\bU_t$, for $t\in \mR$ with eigenvalue $\lambda$. 
\end{definition}

If $\bC$ is a proper invariant subset of $\bX$ in which case $I_\bx=\mR$ for every $\bx\in \bC$, then $\phi$ is called \textcolor{black}{a} subdomain eigenfunction. If $\bC=\bX$, then $\phi$ will be \textcolor{black}{an} ordinary eigenfunction associated with eigenvalue $\lambda$ as defined in (\ref{eig_koopman}). When $\bC$ is open, the open eigenfunctions as defined above can be extended from $\bC$ to a \textcolor{black}{larger set which is the backward-reachable from the closure of $\bC$}, based on the construction procedure outlined in  \cite[Definition 5.2, Lemma 5.1] {mezic2020spectrum}. Following Assumption \ref{assume_sys}, let $\cD$ be the domain of attraction of the equilibrium point at the origin. Our interest is in computing the Koopman eigenfunctions which are defined over this domain $\cD$. Furthermore, these eigenfunctions are associated with the eigenvalues \textcolor{black}{of the dynamic matrix $\bA$ of the linearized system around the equilibrium $\bx=0$}. These principal eigenfunctions are connected to the diffeomorphism as established in the famous Hartman Grobman theorem, which \textcolor{black}{transforms the nonlinear system into} a linear system in a small neighborhood around the equilibrium point \cite{arnold2012geometrical,lan2013linearization}. In fact, these eigenfunctions can be essentially viewed as the extension of the Hartman Grobman diffeomorphism from the local neighborhood around the origin to the entire domain of attraction $\cD$ \cite[Theorem 5.6]{mezic2020spectrum}.

\section{Main Results}

Following Assumption \ref{assume_sys}, we can write the system dynamics (\ref{sys}) as 
\begin{align}
\dot \bx=\bff(\bx)=\bA \bx+\bff_n(\bx),\label{eq:dynamics}
\end{align}
\textcolor{black}{where} $\bA\bx:=\frac{\partial \bff}{\partial \bx}(0)\bx$ is the linear part  and $\bff_n(\bx):=\bff(\bx)-\bA \bx$ is the  purely nonlinear part of the vector field $\bff(\bx)$. Let $\lambda$ be \textcolor{black}{an} eigenvalue of the linearization, i.e., $\bA$, and let $\varphi_\lambda(\bx)$ be the  eigenfunction associated with the eigenvalue $\lambda$ \textcolor{black}{(such eigenfunctions are called principal eigenfunctions)}.
Similar to the system decomposition into linear and nonlinear parts, the principal eigenfunction, $\varphi_\lambda(\bx)$, also admits a decomposition into linear and nonlinear terms as follows:
\begin{align}
\varphi_\lambda(\bx)=\bw_\lambda^\top \bx+h_\lambda(\bx), \label{princ_eig}
\end{align}
where $\bw^\top\bx$ is the linear part and $h_\lambda(\bx)$ is the purely nonlinear term and hence satisfies $\frac{\partial h}{\partial \bx}(0)=0$.
Substituting (\ref{princ_eig}) in  equation (\ref{eig_koopmang})  and comparing terms, we obtain
\begin{align}\bw_\lambda^\top \bA=\lambda \bw_\lambda^\top,\label{linear_eig}
\end{align}
i.e., $\bw_\lambda$ is the left eigenvector of $\bA$ with eigenvalue $\lambda$. Similarly, the nonlinear part, \textcolor{black}{$h_\lambda(\bx)$}, of the eigenfunction satisfies the following linear partial differential equation (PDE) 
\begin{align}\frac{\partial h_\lambda}{\partial \bx}\bff(\bx)-\lambda h_\lambda(\bx)+\bw_\lambda^\top \bff_n(\bx)=0.\label{pde}
\end{align}
The main results of this section on the computation of principal eigenfunctions of the Koopman \textcolor{black}{operator present} an approach for solving equation (\ref{pde}).  We present two different approaches for the computation of the nonlinear part of the principal eigenfunctions. Our first approach relies on the path-integral formula for the computation of principal eigenfunctions. Our second approach relies on the use of a Deep Neural Network for solving the linear PDE (\ref{pde}).

\subsection{Path-Integral Approach for Computation}
Our first results on the path-integral approach for eigenfunction computation provide a solution formula for the linear PDE  (\ref{pde}) using the method of characteristics. 

\begin{theorem}\label{theorem_main} The solution formula for the first order linear PDE (\ref{pde}) can be written as 
\begin{align}
h_\lambda(\bx)=e^{-\lambda t} h_\lambda(\bs_t(\bx))+\int_0^t e^{-\lambda t} \bw^\top_\lambda \bff_n(\bs_\tau(\bx))d\tau,\label{pde_soultion}
\end{align}
where $\bs_t(\bx)$ is the solution of the system (\ref{eq:dynamics}). 
\end{theorem}
\begin{proof}
The PDE (\ref{pde}) can be written as 
\begin{align}
\frac{d h_\lambda(\bs_t(\bx))}{dt}-\lambda h_\lambda(\bs_t(\bx))+\bw^\top_\lambda \bff_n(\bs_t(\bx))=0.
\end{align}
Multiplying throughout by $e^{-\lambda t}$, we obtain
\[\frac{d (e^{-\lambda t} h_\lambda(\bs_t(\bx)))}{dt}+e^{-\lambda t}\bw^\top_\lambda \bff_n(\bs_t(\bx))=0.\]
Next, we integrate the above from $0$ to $t$, \textcolor{black}{thus obtaining}
\begin{gather*}
e^{-\lambda t} h_\lambda(\bs_t(\bx))-h_\lambda(\bx)+\int_0^t e^{-\lambda \tau}h_\lambda(\bs_\tau(\bx))d\tau=0,\\
\implies h_\lambda(\bx)=e^{-\lambda t}h_\lambda(s_t(\bx))+\int_0^t  e^{-\lambda \tau}\bw^\top \bff_n(s_\tau(\bx))d\tau.
\end{gather*}
This completes our proof.
\end{proof}

Our first main result \textcolor{black}{establishes} conditions under which the solution of the PDE (\ref{pde}) is nonlinear. 
\begin{theorem}\label{theorem_main2}
For the dynamical system (\ref{eq:dynamics}) \textcolor{black}{that satisfies Assumption \ref{assume_sys}}, let the origin be an asymptotically stable equilibrium point \textcolor{black}{with the domain of attraction $\cD$ and let $\bA$ be Hurwitz}. Furthermore, all the eigenvalues of the $\bA$ satisfy 
\begin{align}
-{\rm Re}(\lambda)+2{\rm Re}(\lambda_{max})<0,\label{eigendistribution}
\end{align}
where $\lambda_{max}$ is the eigenvalue \textcolor{black}{closest} to the $j\omega$ axis and in the left half plane. Let $h_\lambda$ be the solution of PDE (\ref{pde}) as given in (\ref{pde_soultion}). 
Then, 
\begin{align}
\lim_{t\to \infty}e^{-\lambda t}h_\lambda(\bs_t(\bx))=0,\;\;\;\forall \bx\in \cD\label{condition_convergence}
\end{align}
if 
$h_\lambda(\bx)$ is purely nonlinear function of $\bx$ i.e., $\frac{\partial h_\lambda}{\partial \bx}(0)=0$. 
 \end{theorem}

\begin{proof}
We show that if  $h_\lambda$ is nonlinear then (\ref{condition_convergence}) is true. Since $h_\lambda$ is \textcolor{black}{purely} nonlinear, $\nabla_x h_\lambda(0)=0$ and by construction $h_\lambda(0)=0$. \textcolor{black}{Next,} we show that for every $\varepsilon >0$, there exists $c_{\varepsilon}>0$ such that 
\[
    \|h_\lambda(\bx)\| \le c_\varepsilon \|\bx\|^2
\]
for all $\|\bx\|\le\varepsilon$. 
By applying the mean value theorem inside $\|\bx\|\le\varepsilon$, we have
\begin{gather*}
    h_\lambda(\bx) = h_\lambda(0) + \nabla_\bx h_\lambda(0)\bx + \bx^T\nabla_\bx^2 h_\lambda(\bz)\bx \\
    \implies \|h_\lambda(\bx)\| \le \|\nabla_\bx^2 h_\lambda(\bz)\|\cdot \|\bx\|^2
\end{gather*}
for some point $\bz$ on the line segment joining $0$ and $\bx$. Since $h_\lambda$ is smooth over the compact domain $\|\bx\|\le \varepsilon$, we can define a constant $c_{\varepsilon} := \sup_{\|\bx\|\le\varepsilon} \|\nabla_\bx^2 h_\lambda(\bx)\| $, and obtain the uniform bound 
$\|h_\lambda(\bx)\|\le c_{\varepsilon}\|\bx\|^2$ in the region $\|\bx\|\le\varepsilon$, 
where $c_\varepsilon := \left(\sum_i c_{\varepsilon,i}^2\right)^\frac{1}{2}$.
Now for $\|\bx\|\leq \varepsilon$, there exists, by Hartman Grobman theorem, a near identity change of coordinates with inverse in the small neighborhood around the origin, say of size $\|\bx\|\leq \epsilon$, of the form
\begin{equation}\bz=\bx+\bd(\bx)=\bD(\bx)\iff \bx=\bD^{-1}(\bz)=\bz+\bar \bd(\bz),\label{HG}
\end{equation}
with $\bd(\bx)$ and $\bar \bd(\bz)$ purely nonlinear
such that the nonlinear system is transformed \textcolor{black}{into} linear system i.e.,
$\dot \bx=\bA\bx+\bff_n(\bx)\implies \dot \bz=\bA\bz$ 
and hence 
\[\bs_t(\bx)=\bD^{-1}(e^{\bA t}\bD(\bx))\implies
\bs_t(\bx)=\bD^{-1}(e^{\bA t}(\bx+\bd(\bx)))\]
\[=e^{\bA t}\bx+e^{\bA t}\bd(\bx)+\bar \bd(e^{\bA t}\bx+e^{\bA t}\bd(\bx)).\]
In the above, we have used (\ref{HG}) for $\bD^{-1}$. Since $\bar \bd(\bz)$ is purely nonlinear, for $\|\bx\|\leq \epsilon$, we can get using mean value theorem 
\[\|\bar \bd(\bz)\|\leq c_{\bar d} \|\bz\|^2,\;\;\;\;\|\bd(\bx)\|\leq c_d \|\bx\|^2.\]
Using the above inequality, Cauchy Schwartz inequality,  and the fact that $\|\bx\|\leq \epsilon$, we obtain 
\[\|\bs_t(\bx)\|\leq c_1 e^{{\rm Re}(\lambda_{max} t)}\implies \|\bs_t(\bx)\|^2\leq c_1^2 e^{{\rm Re}(2\lambda_{max} t)} \]
for some constant $c_1$ that depends on $\epsilon, c_d$, and $\bar c_d$. Now
\[ \|h_\lambda(\bs_t(\bx))\|\leq c_\varepsilon \|\bs_t(\bx)\|^2\leq c_2  e^{{\rm Re}(2\lambda_{max} t)},\]
where $c_2 = c_\varepsilon c_1^2$. Then, the limit in equation \eqref{condition_convergence} follows by noting that 
\[\|e^{-\lambda t}h_\lambda(\bs_t(\bx))\|\leq c_2 e^{\left(-{\rm Re}(\lambda)+2{\rm Re}(\lambda_{max})\right)t}.\]

\end{proof}
Using the results of the above theorem we have the following results for the computation of Koopman eigenfunctions under the stability assumption on the system dynamics.


\begin{theorem}\label{theorem_stable}
Consider the dynamical system (\ref{eq:dynamics}) with origin asymptotically stable and with the domain of attraction $\cD$. Let  the eigenvalue $\lambda$ of matrix $\bA$ satisfy condition (\ref{eigendistribution}). Then the principal eigenfunction, $\phi_\lambda$, corresponding to eigenvalue $\lambda$, is well defined in the domain $\cD$ and is given by following path-integral formula:
\begin{align}
\phi_\lambda(\bx)=\bw_\lambda^\top \bx+\int_0^\infty e^{-\lambda t}\bw_\lambda^\top \bff_n(\bs_t(\bx))dt
\end{align}
where $\bw_\lambda$ \textcolor{black}{satisfies} $\bw_\lambda^\top \bA=\lambda \bw_\lambda^\top$. 
\end{theorem}
\begin{proof}
The eigenfunction corresponding to eigenvalue $\lambda$ admits a decomposition into linear and nonlinear parts as given in Eqs. (\ref{princ_eig}) and (\ref{linear_eig}). Since $h_\lambda$ is assumed to be nonlinear, the results of Theorem \ref{theorem_main2} applies and hence $\lim_{t\to \infty}e^{-\lambda t}h_\lambda(\bs_t(\bx))=0$ for all $\bx\in \cD$. The \textcolor{black}{result} then follows by applying Theorem \ref{theorem_main} on the solution formula of linear PDE.
\end{proof}
\begin{remark} 
The eigenfunctions $\phi_{\lambda_i}$ for $i=1,\ldots,n$ can be used as diffeomorphism for the linearization of nonlinear system valid within the domain of attraction $\cD$. In \cite{lan2013linearization,mezic2020spectrum}, the authors propose an approach for the construction of such diffeomorphism valid within the domain of attraction based on the extension of the Hartman Grobman diffeomorphism, which is known to exist in a small neighborhood of the origin. 
\end{remark}
The results of Theorem \ref{theorem_stable} can be extended to compute the Koopman spectrum for the system with linearization having all its eigenvalues in the right half plane by time \textcolor{black}{reversing} the vector field. We have the following Corollary in this direction.

\begin{corollary}\label{corollary_unstable}
Consider the dynamical system (\ref{eq:dynamics}) satisfying Assumption \ref{assume_sys}. Let the matrix $\bA$ for the linearization of system dynamics \textcolor{black}{have} all its eigenvalues in the strict right half plane with eigenvalue, $\lambda$, satisfying the condition
\begin{align}
{\rm Re}(\lambda)-2{\rm Re}(\lambda_{max})<0\label{eigendistribution_unstable}.
\end{align}
 The principal eigenfunction, $\phi_\lambda$, with eigenvalue $\lambda$, are well defined in the domain $\bar\cD:=\{\bx\in \bX: \lim_{t\to \infty} \bs_{-t}(\bx)=0\}$ and is given by the following formula
\begin{align}
\phi_\lambda(\bx)=\bw_\lambda^\top \bx+\int_0^\infty e^{\lambda t}\bw_\lambda^\top \bff(\bs_{-t}(\bx))dt,\label{path_integral_backward}
\end{align}
where $\bw_\lambda$ \textcolor{black}{satisfies} $\bw_\lambda^\top \bA=\lambda \bw_\lambda^\top$. 
\end{corollary}

Theorem \ref{theorem_stable} and Corollary \ref{corollary_unstable} \textcolor{black}{provide} an approach for computing the Koopman principal eigenfunctions for the cases when the equilibrium point is stable and anti-stable. 
It is important to emphasize that the results of Theorem \ref{theorem_stable} and Corollary \ref{corollary_unstable} rely on the \textcolor{black}{sufficient} condition that can be verified for the computation of principal eigenfunction. \textcolor{black}{The following theorem for principal eigenfunction computation applies to a system with a saddle-type equilibrium point.}
\begin{theorem}\label{theorem_saddle} Consider the dynamical system (\ref{eq:dynamics}) satisfying Assumption \ref{assume_sys} with $\lambda$ \textcolor{black}{as an} eigenvalue of $\bA$ such that ${\rm Re}(\lambda)>0$. Assume that $h_\lambda(\bx)$, the nonlinear part of the principal eigenfunction corresponding to eigenvalue $\lambda$ satisfy 
\begin{align}\lim_{t\to \infty}|h_\lambda(\bs_t(\bx))|\leq M\label{bounded}
\end{align}
for some constant $M$ and for all \textcolor{black}{$\bx$ in some set $\bX_1 \subseteq \bX$}. Then the eigenfunction corresponding to eigenvalue $\lambda$ can be computed using the following path-integral formula for all $\bx\in \bX_1:$
\begin{align}
\phi_\lambda(\bx)=\bw_\lambda^\top \bx+\int_0^\infty e^{-\lambda t}\bw_\lambda^\top \bff(\bs_t(\bx))dt.\label{pathintegral}
\end{align}
\end{theorem}
\begin{proof}
The condition (\ref{bounded}) combined with the fact that ${\rm Re}(\lambda)>0$ ensure that $\lim_{t\to \infty}e^{-\lambda t}h_\lambda(\bs_t(\bx))=0$. \textcolor{black}{The expression \eqref{pathintegral} then follows from \eqref{princ_eig} and the PDE solution \eqref{pde_soultion} by taking $t\rightarrow 0.$}
\end{proof}
Note that the main issue with applying the results from the above Theorem is that the condition (\ref{bounded}) cannot be easily verified. For a system with saddle-type equilibrium point, \textcolor{black}{computation of} eigenfunctions corresponding to eigenvalues with negative real part can be \textcolor{black}{similarly done} by applying the results of Theorem \ref{theorem_saddle} for the time-reversed vector field. We would like to emphasize that in applications such as optimal control, it is of interest to compute only part of eigenfunctions corresponding to unstable eigenvalues \cite{vaidya_Koopmanspectrumcdc2022}

\subsection{Deep Neural Network for Principal Eigenfunction}\label{subsec:DNN}
Deep learning techniques have been successfully applied in literature  towards computation of the Koopman operator and its associated eigenfunctions \cite{lusch2018deep, folkestad2020extended}. In all of these prior works, the main approach is to parameterize the eigenfunctions (or nonlinear `lifting' functions in other cases) using  autoencoders and then utilizing sampled trajectory data to compute the loss function for training.

Let $\mathcal{P}=\left\{(\bx_i,\mathbf{y}_i)\right\}_{i\in\mathcal{I}}$ be a set of points along system trajectories sampled at a uniform time interval $\tau$, that is,
\[
    \by_i = \bs_\tau(\bx_i), \;i\in\mathcal{I}.
\] Then, the DNN parameterized vector of eigenfunctions or lifting functions $\psi_\theta$ is typically learned by minimizing the loss
\begin{equation}\label{eq:loss_others}
    \min_{K,\theta,\omega} \left[ 
    \begin{array}{c}
\underset{(\bx,\by)\sim\mathcal{D}}{\mathbb{E}} \big[ \|\psi_\theta(\by) - K\psi_\theta(\bx)\|\big] \;+ \\
     
     \underset{x\sim \bX}{\mathbb{E}}\big[ \|\bx - \eta_\omega\big(\psi_\theta(\bx)\big)\|\big] \
     \end{array}
     \right],
\end{equation}
where $\mathbb{E}[\cdot]$ denotes the expected value with respect to the data distribution specified. The function $\eta_\omega$ is a decoder network parameterized by $\omega$, which maps points from the lifted Koopman space back to the original state-space 
 and $K$ is the finite-dimensional approximation of the Koopman operator. The second term in the equation above is the auto-encoder loss and is needed to ensure that the DNN does not learn a trivial solution $\psi_\theta \equiv 0$. In place of the first term, it is also common to use Koopman PDE \eqref{eq:koop_def} in the loss function, wherein one penalizes the violation in the PDE satisfaction. In the case where the DNN parameterizes the lifting function, one needs to indirectly extract the eigenfunctions using the learned $K$ matrix and $\psi_\theta$.

Our approach using path-integral can be used to learn the principal Koopman eigenfunctions in a more direct fashion, using the equation \eqref{princ_eig} to create a labeled training dataset $\mathcal{D}' = \left\{(\bx_i,\phi_\lambda(\bx_i))\right\}_{i\in\mathcal{I}}$, thus leading to the following supervised learning problem:
\begin{equation}
\min_{\theta} \underset{(\bx,\bz)\sim\mathcal{D}'}{\mathbb{E}} \left[\|\bz - \bw^\top\bx - \hat{h}_\theta(\bx)\|\right], \label{eq:loss1}
\end{equation}
where $\theta$ parameterizes the nonlinear part of the principal eigenfunction using the DNN $\hat{h}_\theta$. Additionally, one can introduce the following secondary term in the loss function for regularization:
\begin{equation}\label{eq:loss2}
    \underset{\bx \in X}{\mathbb{E}}\left[ \Big\| \frac{\partial \hat{h}_\theta}{\partial \bx}\bff(\bx)-\lambda \hat{h}_\theta(\bx)+\bw^\top \bff_n(\bx)\Big\|\right].
\end{equation}
This ensures that the network does not overfit to the dataset $\mathcal{D}'$. Note that this secondary term \eqref{eq:loss2} is much cheaper to evaluate compared to the loss term in \eqref{eq:loss1} due to offline computations involved in the generation of labeled dataset $\mathcal{D}'$. Moreover, since PDE \eqref{pde} does not admit a trivial solution (unlike PDE \eqref{eq:koop_def}), we do not need an additional auto-encoder loss term like in equation \eqref{eq:loss_others}.

\section{Simulation Results}

\noindent{\bf Analytical Example 1}: Consider the dynamics of a one-dimensional system given by 
\begin{align}
    \dot{x} = \alpha(x-x^3). \nonumber
\end{align}
The principal eigenfunctions for this system can be computed analytically as $\phi(x) = \frac{x}{\sqrt{1-x^2}}$. Note that $\phi(x)$ is well-defined within the domain $x \in (-1,1)$. For $\alpha=-1$, the system has a stable equilibrium point at the origin (with eigenvalue $\lambda=-1$). Although $\phi(x)$ blows up as $x \xrightarrow{}(-1,1)$, since $\bs_t(x)\xrightarrow{}0$, condition in Eq.  \eqref{condition_convergence} is satisfied. The corresponding eigenfunction can be estimated using Theorem \ref{theorem_stable} as shown in Fig. \ref{fig:analytical_1D}a. For $\alpha=1$, the origin is unstable, and hence the results of Theorem \ref{theorem_stable} do not apply.  But the results of Corollary \ref{corollary_unstable} apply, and the estimated eigenfunction using Eq. (\ref{path_integral_backward}) matches perfectly with the analytical solution.\\
\begin{figure}
    \centering
    \includegraphics[width = \linewidth]{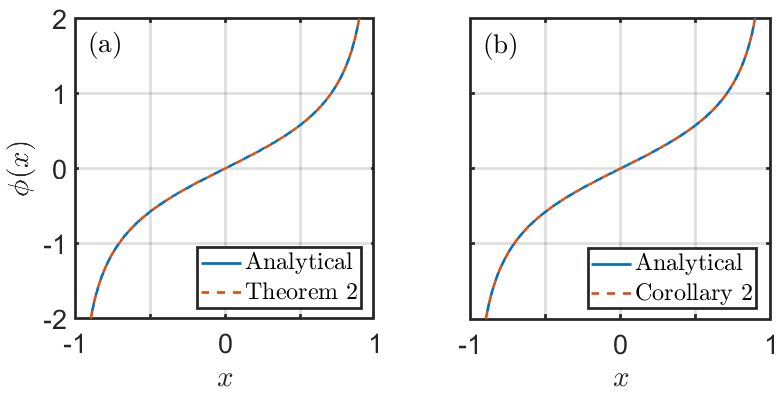}
    \caption{Analytical example 1: (a) {\color{black}eigenfunction corresponding to stable eigenvalue} estimated using Theorem \ref{theorem_stable} (b) {\color{black} eigenfunction corresponding to unstable eigenvalue} estimated using Corollary \ref{corollary_unstable}.}
    \label{fig:analytical_1D}
\end{figure}

\noindent{\bf Analytical Example 2:} Consider the dynamics of a two-dimensional system given by 
\begin{align}
 &\dot{x}_1 = -2\lambda_2x_2(x_1^2-x_2-2x_1x_2^2+x_2^4) \nonumber \\
 &+\lambda_1(x_1+4x_1^2x_2-x_2^2-8x_1x_2^3+4x_2^5) \nonumber\\
 &\dot{x}_2 = 2\lambda_1(x_1-x_2^2)^2-\lambda_2(x_1^2 x_2 - 2x_1 x_2^2 + x_2^4) \nonumber
\end{align}
where $\lambda_1, \lambda_2$ are the eigenvalues of the system {\color{black}when linearized about the origin} \cite{bollt2018matching}. For this system, the eigenfunctions can be computed analytically as $\phi_{\lambda_1}(x) = x_1-x_2^2$ and $\phi_{\lambda_2}(x) = -x_1^2+x_2+2x_1x_2^2-x_2^4$. We pick the eigenvalues $\lambda_1=-1$ and $\lambda_2=3$ such that the system has a saddle equilibrium at the origin. The analytical eigenfunction corresponding to $\lambda_2=3$ is shown in Fig. \ref{fig:analytical_2D}a.  The {\color{black} eigenfunction corresponding to the unstable eigenvalues} can be estimated accurately using Theorem \ref{theorem_saddle} as shown in Figure \ref{fig:analytical_2D}b.\\ 
\begin{figure}
    \centering
    \includegraphics[width = \linewidth]{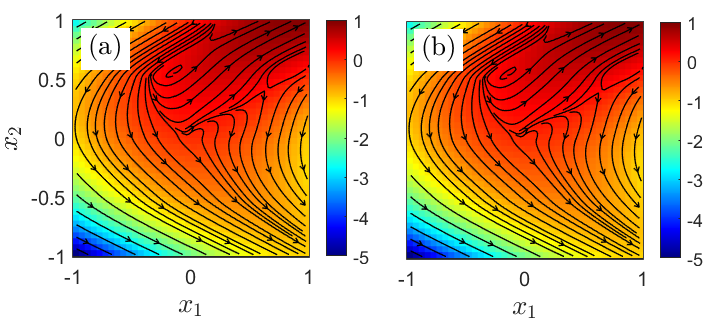}
    \caption{Analytical Example 2 with saddle equilibrium point: Eigenfunction corresponding to ${\rm Re}(\lambda)>0$. (a) analytical (b) estimated using Theorem \ref{theorem_stable}.}
    \label{fig:analytical_2D}
\end{figure}

\noindent {\bf Duffing Oscillator}: The Duffing oscillator dynamics is
\begin{align*}
    &\dot{{x}}_1 = {x}_2, \;\;\;\; \dot{{x}}_2 = {x}_1 - \delta {x}_2 - {x}_1^3
\end{align*}
For eigenfunction computation, we use $\delta = 0.5$. The equilibrium point at the origin is a saddle point. Fig. \ref{fig:duffing_eigfun}a shows the {\color{black} eigenfunction corresponding to the unstable eigenvalue} obtained for the equilibrium point at the origin after $t=20$ s using Theorem \ref{theorem_saddle}. Since the eigenfunctions remain bounded, equation \eqref{condition_convergence} is satisfied. The stable manifold (shown in yellow in Fig. \ref{fig:duffing_eigfun}b) is obtained as the zero-level set of this eigenfunction. The magnitude of the {\color{black} (complex) eigenfunction corresponding to the stable eigenvalue}   obtained after $t=20 s$ for the equilibrium point at [1,0] is shown in Fig. \ref{fig:duffing_eigfun}c. 

The Lyapunov function verifying the stability of the equilibrium dynamics is constructed as $V(\bx)=\Phi^\top(\bx) \bP\Phi(\bx)$, where $\bP$ is a positive matrix obtained as the solution of the following Lyapunov equation $\Lambda^\top \bP+\bP\Lambda<0$ \cite{mauroy2016global}. The Lyapunov function for this system is shown in Fig. \ref{fig:duffing_eigfun}d.\\ 

\begin{figure}
    \centering
    \includegraphics[width = \linewidth]{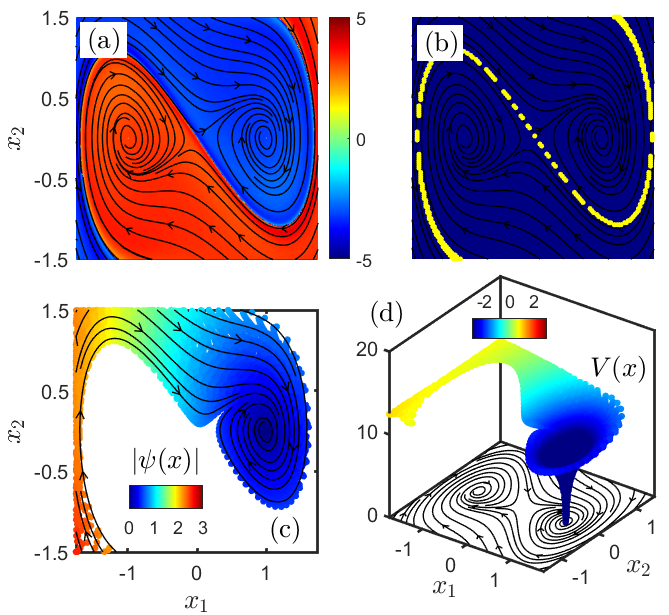}
    \caption{Duffing Oscillator: (a) eigenfunction (real) for $\lambda = 0.78$ at the (0,0); (b) zero level set representing the stable manifold; (c) magnitude of the 
    eigenfunction (complex) for $\lambda = -0.25 \pm 1.39j$ at $(1,0)$; (d) Lyapunov function obtained  from (c)}
    \label{fig:duffing_eigfun}  
\end{figure}



\noindent {\bf Two Link Robotic Arm:} Consider the following Euler-Lagrange dynamics representing a 2-link manipulator:
\begin{equation}\label{eq:EL}
    \bM(\bq)\ddot{\bq} + \mathbf{C}(\bq,\dot{\bq})\dot{\bq} + \bG(\bq) = \bB\dot{\bq} 
\end{equation}
where $\bq\in \mathbb{R}^2$ represents the generalized coordinates of the manipulator. Specifically, we take 
\[
   \bM(\bq) = \left[\begin{array}{cc}
       2\cos(q_2) + 8.33 &  \cos(q_2) + 0.33\\
       \cos(q_2) + 0.33 & 0.33
   \end{array}\right]
\]
\[
\mathbf{C}(\bq,\dot{\bq}) = \left[\begin{array}{cc}
       -2\dot{q}_2\sin(q_2) &  -\dot{q}_2\sin(q_2)\\
       \dot{q}_1\sin(q_2)& 0
   \end{array}\right]
\]
\[
   \bG(\bq) = \left[\begin{array}{c}
       50\sin(q_1) + 5\sin(q_1 + q_2)\\
       5\sin(q_1 + q_2)
   \end{array}\right]
\]
and $\bB = diag [5.5, \quad 0.001]$, where $diag$ represents a diagonal matrix. We take the $4$-dimensional state to be $\bx = [q_1,q_2,\dot{q}_1,\dot{q}_2]$, and rewrite the dynamics \eqref{eq:EL} in standard form as $
    \dot{\bx} = \mathbf{f}(\bx).
$
The linearized system about the stable equilibrium $\bx=0$ has complex eigenvalues $\lambda_{1,2} = -0.23 \pm 2.29j$ and $\lambda_{3,4} = -0.32 \pm 5.32j$, thus leading to complex eigenfunctions. We pick a domain $[-\frac{\pi}{12}, \frac{\pi}{12}]^4$ over which we compute the path integrals and create a dataset $\mathcal{D}'$ \textcolor{black}{as described in Subsection \ref{subsec:DNN}}. This dataset, along with the sum of losses \eqref{eq:loss1} and \eqref{eq:loss2}, is then used to train a multi-layer perceptron network (MLP) with a sinusoidal activation function. The MLP has 3 hidden layers, each with 128 neurons. The input layer is of size 4, and the output layer has a size 2, corresponding to the real and imaginary parts of the eigenfunction being learned. Fig. \ref{fig:eigfun_manipulator} shows the magnitude and phase of the complex eigenfunction along the system trajectory starting at random initial conditions within the domain. It can be seen that the magnitude of the eigenfunction goes to zero along the stable trajectory.
\begin{figure}
    \centering
    \includegraphics[width = \linewidth]{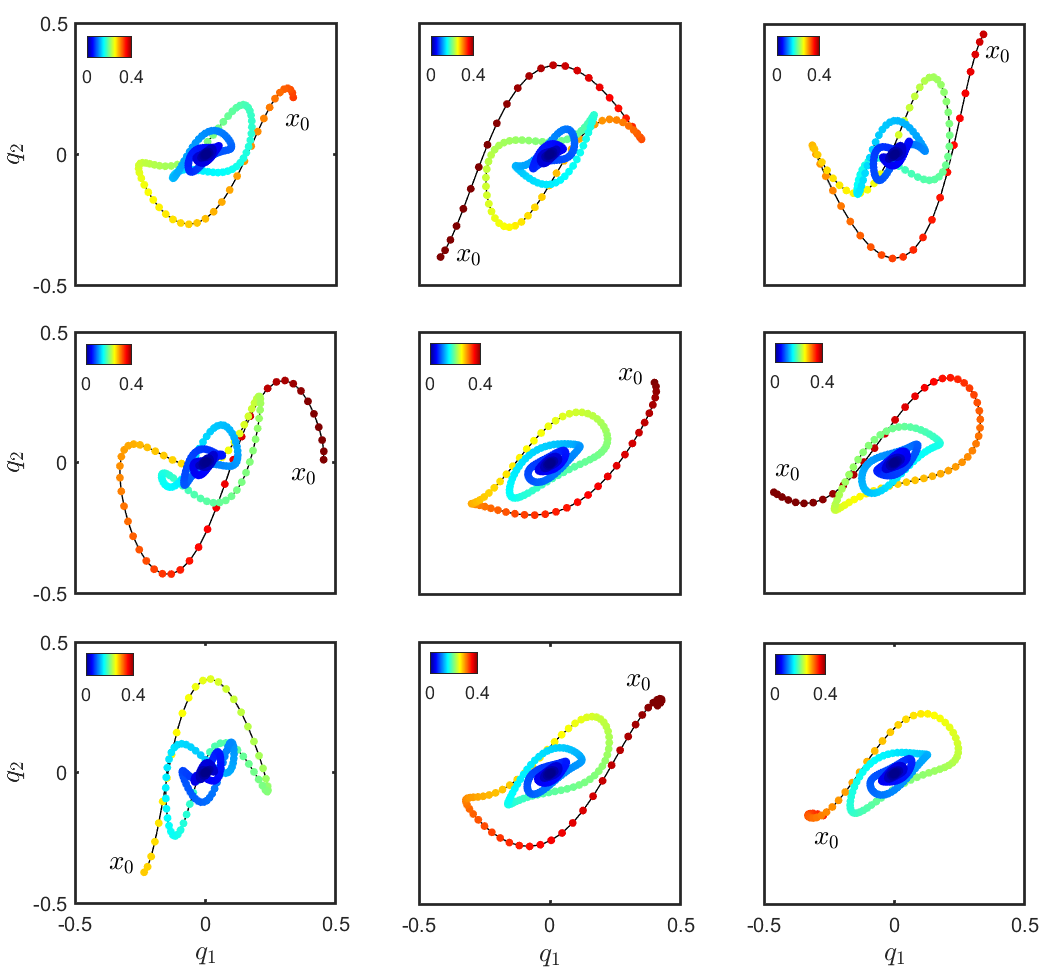}
    \caption{Magnitude of the \textcolor{black}{eigenfunction corresponding to stable eigenvalue} $\lambda = -0.23 \pm 2.29j$ along the trajectories of the system starting from random initial conditions}
    \label{fig:eigfun_manipulator}
\end{figure}

\section{Conclusions}
We provide a novel approach for the computation of principal eigenfunctions of the Koopman operator based on the path-integral formula. Furthermore, the path-integral formula is used to formulate the DNN-based approach for computing the eigenfunctions. Simulation results show that the path-integral-based approach accurately approximates the principal eigenfunctions of systems with complex dynamics. We demonstrate the applications of eigenfunctions for the computation of stable/unstable manifolds and the Lyapunov function. Simulation results involving analytical examples, duffing oscillator, and two links robotic arm are presented to show the efficacy of the developed framework. Future research will focus on a data-driven approach for the computation of principal eigenfunctions and its extension to discrete-time dynamical systems. 

\bibliographystyle{IEEEtran}
\bibliography{references,ref1,Umesh_ref}


\end{document}